\documentclass[11pt]{article}

\usepackage{amsmath,amssymb,amsthm,amsfonts,enumerate,graphicx}
\usepackage{amssymb, amsfonts}
\usepackage{amsmath}
\usepackage{latexsym}
\usepackage{mathrsfs}
\usepackage{graphicx}
\usepackage{enumerate}
\newtheorem{theorem}{Theorem}[section]
\newtheorem{lemma}[theorem]{Lemma}
\newtheorem{proposition}[theorem]{Proposition}
\newtheorem{corollary}[theorem]{Corollary}
\newtheorem{definition}{Definition}[section]

\newtheorem{example}[theorem]{Example}
\numberwithin{equation}{section}

\renewcommand{\baselinestretch}{1.0}

\begin{document}

\title{Measure of Self-Affine Sets and Associated Densities \footnotetext{Math Subject Classifications. 28A78, 28A80.}\footnotetext{Keywords. Self-affine sets,  Lebesgue measure,  Hausdorff measure, upper Beurling density.}
\footnotetext{Email:xiaoyefu@gmail.com (X. Y. Fu), gabardo@mcmaster.ca (J. P. Gabardo).} \footnotetext{Tel: 00852-3943-1984 (X. Y. Fu), 1-905-525-9140, ext. 23416 (J. P. Gabardo).}}
\author{Xiaoye Fu, Jean-Pierre Gabardo\\
    \small Department of Mathematics, The Chinese University of Hong Kong, Hong Kong\\
     \small Department of Mathematics and Statistics, McMaster University, Hamilton, O.N.\\
     \small  L8S 4K1, Canada}

\date{}
\maketitle

\begin{abstract}
Let $B$ be an $n\times n$ real expanding matrix and $\mathcal{D}$ be a finite subset of $\mathbb{R}^n$ 
with $0\in\mathcal{D}$. The self-affine set $K=K(B,\mathcal{D})$ is the unique compact set satisfying the 
set-valued equation $BK=\displaystyle\bigcup_{d\in\mathcal{D}}(K+d)$. In the case where 
$\text{card}(\mathcal{D})=\lvert\det B\rvert,$ we relate the Lebesgue measure of $K(B,\mathcal{D})$ to the 
upper Beurling density of the associated measure 
$\mu=\lim\limits_{s\to\infty}\sum\limits_{\ell_0,\dotsc,\ell_{s-1}\in\mathcal{D}}\delta_{\ell_0+B\ell_1+\dotsb+B^{s-1}\ell_{s-1}}.$ 
If, on the other hand, $\text{card}(\mathcal{D})<\lvert\det B\rvert$ and $B$ is a similarity matrix, we relate the Hausdorff measure 
$\mathcal{H}^s(K)$, where $s$ is the similarity dimension of $K$, to a corresponding notion of upper density for the measure $\mu$.
\end{abstract}

\renewcommand{\baselinestretch}{1.0}

\section{Introduction}
 Let $M_n(\mathbb{R}) \ (M_n(\mathbb{Z}))$ be the set of $n\times n$ matrices with real (integer) entries. 
Let $B\in M_n(\mathbb{R})$ be an \emph{expanding matrix}, i.e. all its eigenvalues $\lambda_i$ satisfy 
$\lvert\lambda_i\rvert>1$ and let $\mathcal{D}\subseteq\mathbb{R}^n$ be a finite set of 
distinct real vectors with $0\in\mathcal{D}$. We call $\mathcal{D}$ a \emph{digit set} and $(B,\mathcal{D})$ a \emph{self-affine pair}.
Let 
$$f_d(x)=B^{-1}(x+d), \ d\in\mathcal{D}.$$ An important property of these maps is that they are contractive with respect to a suitable norm on $\mathbb{R}^n$ 
(see \cite{LWG}). The family of mappings $\{f_d(x)\}_{d\in\mathcal{D}}$ is called an \emph{iterated function system} (IFS). It is well-known that there exists a unique non-empty compact set $K:=K(B,\mathcal{D})$ satisfying $K=\bigcup\limits_{d\in\mathcal{D}}f_d(K)$, or equivalently, $BK=\displaystyle\bigcup_{d\in\mathcal{D}}(K+d).$
We call this set $K$ the \emph{self-affine set} determined by the self-affine pair $(B,\mathcal{D})$. Define
$$\mathcal{D}_s:=\Big\{\displaystyle\sum_{j=0}^{s-1}B^j\ell_j:\ell_j\in\mathcal{D},j\geqslant 0\Big\} 
\ {\rm{for}} \ s\geqslant 1 \ {\rm{and}} \ \mathcal{D}_{\infty}:=\displaystyle\bigcup_{s=1}^{\infty}\mathcal{D}_s.$$
The inclusion $\mathcal{D}_s\subset\mathcal{D}_{s+1}$ holds for any $s\ge 1$ since $0\in\mathcal{D}$.

The situation in which $\text{card}(\mathcal{D})=\lvert\det B\rvert\in\mathbb{Z}$, where $\text{card}(\mathcal{D})$ denotes the cardinality of $\mathcal{D}$,
has been studied extensively. In this case, 
$K(B,\mathcal{D})\subset\mathbb{R}^n$ is called a \emph{self-affine tile} if it has positive Lebesgue measure. 
 Lagarias and Wang proved the following result.
\begin{theorem}[\cite{LWG}]\label{t-2-1}
Suppose that $\text{card}(\mathcal{D})=\lvert\det B\rvert=m\in\mathbb{Z}$. Then the set
 $K(B,\mathcal{D})$ has positive Lebesgue measure if and only if for each $k\ge 1$, all $m^k$ expansions in $\mathcal{D}_k$ 
are distinct, and $\mathcal{D}_{\infty}$ is a uniformly discrete set, i.e. there exists $\delta>0$ such that 
$\lVert x-y\rVert>\delta$ for any $x\ne y\in \mathcal{D}_{\infty}$. 
\end{theorem}

Many aspects of the theory of self-affine tiles have been investigated thoroughly.
Among them, let us mention  the structure and tiling properties, the connection to wavelet theory, 
the fractal structure of the boundaries and the classification of tile digit sets 
(see e.g. \cite{LWG, LWN, GHS, LWA, GM, LW, HL, LLR2012, LLR2013}). However, there is a basic question remaining unsolved. It is the following:

{\bf Question.} What is $\lvert K(B,\mathcal{D})\rvert$ with $\text{card}(\mathcal{D})=\lvert\det B\rvert$ and 
$\mathcal{D}\subset\mathbb{R}^n$, where $\lvert K\rvert$ denotes the Lebesgue measure of a measurable set $K$?

Regarding this question, it is well-known \cite{LWN} that $\lvert K(B,\mathcal{D})\rvert$ is a positive integer when $B\in M_n(\mathbb{Z})$ 
is an expanding matrix and $\mathcal{D}\subset\mathbb{Z}^n$ is a complete set of coset representatives for $\mathbb{Z}^n/B\mathbb{Z}^n$. 
Gabardo and Yu \cite{GY} provided an algorithm to evaluate the Lebesgue measure of such self-affine tiles. 
One of our goal here is to relate, for more general $\mathcal{D}\subset\mathbb{R}^n$, 
the number $\lvert K(B,\mathcal{D})\rvert$ to the upper Beurling density of an associated measure $\mu$, which is defined by
\begin{eqnarray}\label{e-1-2}
\mu=\lim\limits_{s\to\infty}\sum\limits_{\ell_0,\dotsc,\ell_{s-1}\in\mathcal{D}}\delta_{\ell_0+B\ell_1+\dotsb+B^{s-1}\ell_{s-1}},
\end{eqnarray}
where $\delta_x$ denotes the Dirac measure at $x$.
In particular, we will prove the following result.
\begin{theorem}\label{Lebesgue}
Let $B\in M_n(\mathbb{R})$ be an expanding matrix with $\lvert\det B\rvert\in\mathbb{Z}$ and let $\mathcal{D}\subset\mathbb{R}^n$
 be a finite set with $\text{card}(\mathcal{D})=\lvert\det B\rvert$. Then, 
$\lvert K(B,\mathcal{D})\rvert^{-1}$ is equals to the upper Beurling density of $\mu$, where $\mu$ is defined by (\ref{e-1-2}).
\end{theorem}

For the case $\text{card}(\mathcal{D})>\lvert\det B\rvert$, the situation becomes more complicated because 
the sets $K+d$, $d\in\mathcal{D}$, might overlap. 
He, Lau and Rao (\cite{HLR}) considered the problem as to whether or not $\lvert K(B,\mathcal{D})\rvert$ is positive for this case. 

It is easy to see that $\lvert K(B,\mathcal{D})\rvert=0$ if $\text{card}(\mathcal{D})<\lvert\det B\rvert$. 
However, if we replace the ordinary Euclidean dimension $n$ with the Hausdorff dimension, then 
some self-affine sets with zero Lebesgue measure may have positive Hausdorff measure associated with their associated Hausdorff dimension. 
For example, all self-similar sets satisfying the open set condition have positive Hausdorff measure.

$K(B,\mathcal{D})$ is called a \emph{self-similar set} if the matrix $B=\rho R$, where $\rho>1$ and $R$ is an orthogonal matrix. In this case, the matrix $B$ is called a \emph{similarity matrix} with scaling factor $\rho>1$. 
We say that the IFS $\{f_d\}_{d\in\mathcal{D}}$ satisfies the \emph{open set condition} (OSC) if there exists a non-empty bounded open set $V$ such that 
\begin{eqnarray*}
\bigcup\limits_{d\in\mathcal{D}} f_d(V)\subset V \ {\rm{and}} \ f_d(V)\bigcap f_{d^{\prime}}(V)=\emptyset \ {\rm{for}} \ d\ne d^{\prime}\in\mathcal{D}.
\end{eqnarray*}

In the following, we denote the Hausdorff dimension of a measurable set $K\subset\mathbb{R}^n$ as $\dim_H K$ and the 
Hausdorff measure associated with its Hausdorff dimension $s:=\dim_H K$ as $H^{s}(K)$. The problem of  computing the 
Hausdorff dimension or the Hausdorff measure of a self-affine set has intrigued many researchers for a long time. 
A well-known result on this topic was given in \cite{F}.
\begin{theorem}[\cite{F}]\label{t-3-1}
Let $B$ be a similarity matrix with scaling factor $\rho>1$. Suppose that the IFS $\{f_d(x)\}_{d\in\mathcal{D}}$ satisfies the OSC. 
Then the Hausdorff dimension of $K:=K(B,\mathcal{D})$ is $s:=\dim_H K=\log_{\rho}^{\text{card}(\mathcal{D})}$ and $0<\mathcal{H}^s(K)<\infty$. 
\end{theorem}
The number $s=\log_{\rho}^{\text{card}(\mathcal{D})}$ in Theorem \ref{t-3-1}
is called the \emph{similarity dimension} of the self-similar set $K(B,\mathcal{D})$. 

Even if we assume that a self-similar set $K:=K(B,\mathcal{D})$ satisfies the OSC, it might still be difficult to compute $\mathcal{H}^s(K)$
exactly. In \cite{AS}, 
Ayer and Strichartz provided an algorithm to compute the Hausdorff measure of a 
class of linear Cantor sets in dimension one.
 However, no similar result exists in higher dimension, even for self-similar sets. Some estimates on the Hausdorff measure 
of Sierpinski carpet and Sierpinski gasket, which are a special class of self-similar sets, can also be found 
in \cite{Xiong2005, J2007, JZZ2002, ZF2000, ZL2000}. 
 We will provide here an analogue of Theorem \ref{Lebesgue} in the form of a relation between
the $s$-Hausdorff measure of a self-similar set and the quantity $\mathcal{E}_s^+(\mu)$, which is the upper $s$-density of the measure $\mu$ defined
in (\ref{e-1-2}) (see Definition \ref{d-3-1}), where $s$ is the similarity dimension of $K$.
\begin{theorem}\label{hausdorff}
Let $B$ be a similarity matrix with scaling factor $\rho>1$ and let $\mathcal{D}\subset\mathbb{R}^n$ be a finite set with 
$\text{card}(\mathcal{D})\le\lvert\det B\rvert$. Suppose that $K$ is the self-similar set determined by the pair $(B,\mathcal{D})$
and $s$ is the associated similarity dimension.
 Then, $\mathcal{H}^s(K)=(\mathcal{E}_s^+(\mu))^{-1}$.
\end{theorem}

We should remark that the upper $s$-density of $\mu$, $\mathcal{E}_s^+(\mu)$, is not easy to compute in general. 
As an application of Theorem \ref{hausdorff}, we will show how
 to compute the Hausdorff measure of a class of Cantor sets at the end of this paper.

The paper is organized as follows. In Section 2, we collect some known results on  Beurling densities and Hausdorff measures
that we will use in Section 3 and in Section 4.
 In Section 3, we provide some applications of the notion of Beurling density to the geometric structure of self-affine tiles. 
In particular, we consider the problem of computing the Lebesgue measure of the self-affine set 
$K(B,\mathcal{D})$ when $\text{card}(\mathcal{D})=\lvert\det B\rvert$ and 
relate it to the upper Beurling density of $\mu$, where $\mu$ is defined by (\ref{e-1-2}). In Section 4, we consider the case where 
$\text{card}(\mathcal{D})<\lvert\det B\rvert$ and $B$ is a similarity with scaling factor $\rho>1$.
We develop there the main tools to prove Theorem \ref{hausdorff}.
 Finally, using this last result, we compute the Hausdorff measure of a class of Cantor sets in Section 5.

\section{Preliminaries}

In this section, we introduce the notion of upper-Beurling (resp.~lower-Beurling) density of a positive measure
and recall the definition of Hausdorff measures. 
We collect some known results on the properties of Beurling densities that we will 
use in Section 3 and others concerning the OSC and  Hausdorff measures.

Let $\mu$ be a positive Borel measure in $\mathbb{R}^n$. The \emph{upper Beurling density}, $D^+(\mu)$, and 
the \emph{lower Beurling density}, $D^-(\mu)$, of $\mu$ are defined respectively by 
\begin{eqnarray*}
D^+(\mu)=\limsup\limits_{N\rightarrow\infty}\sup\limits_{z\in\mathbb{R}^n}\frac{\mu(I_N(z))}{N^n}, \ 
D^-(\mu)=\liminf\limits_{N\rightarrow\infty}\inf\limits_{z\in\mathbb{R}^n}\frac{\mu(I_N(z))}{N^n},
\end{eqnarray*}
where $I_N(z)=\Big\{y=(y_1,\dotsc,y_n)\in\mathbb{R}^n, \lvert y_i-z_i\rvert\le \frac{N}{2},i=1,\dotsc,n\Big\}$. 
If $D^+(\mu)=D^-(\mu)$, we say that the Beurling density of the measure $\mu$ exists and we denote it by $D(\mu)$.

If $\Lambda\subset\mathbb{R}^n$ is a discrete subset, we define $D^+(\Lambda):=D^+(\mu)$ and $D^-(\Lambda):=D^-(\mu)$
where $\mu=\sum\limits_{\lambda\in \Lambda}\delta_{\lambda}$. The quantities 
$D^+(\Lambda)$ and $D^-(\Lambda)$ are called the \emph{upper and the lower Beurling density} of $\Lambda$,
 respectively. More explicitely, we have
\begin{eqnarray*}
D^+(\Lambda)=\limsup\limits_{N\rightarrow\infty}\sup\limits_{z\in\mathbb{R}^n}\frac{\text{card}(\Lambda\bigcap I_N(z))}{N^n}, \ 
D^-(\Lambda):=\liminf\limits_{N\rightarrow\infty}\inf\limits_{z\in\mathbb{R}^n}\frac{\text{card}(\Lambda\bigcap I_N(z))}{N^n}.
\end{eqnarray*}
If $D^+(\Lambda)=D^-(\Lambda)$, then we say that $\Lambda$ has uniform Beurling density and we denote this density by $D(\Lambda)$ (see \cite{CKS}). 

Gabardo \cite{G} established a connection between certain convolution inequalities for positive Borel measures in $\mathbb{R}^n$ and the corresponding notions of  Beurling density associated with such measures. Let $f\in L^1(\mathbb{R}^n)$ with $f\ge 0$ and let $\mu$ be a positive Borel measure on $\mathbb{R}^n$ which is finite on compact sets. The convolution $f*\mu$ is defined by
\begin{eqnarray*}
\int_{\mathbb{R}^n}\,\varphi(t)\,d(f*\mu)(t)=\int_{\mathbb{R}^n}\int_{\mathbb{R}^n}\,\varphi(x+y)\,f(y) \, dy\, d\mu(x),
\end{eqnarray*}
where $\varphi\in C_c^+(\mathbb{R}^n)$ (the space of non-negative continuous functions with compact support on $\mathbb{R}^n$). In the following, 
we will list some of the results from  \cite{G} for later use.
Recall that a positive Borel measure $\mu$ on $\mathbb{R}^n$ is called \emph{translation-bounded} if, 
for every compact set $K\subset\mathbb{R}^n$, there exists a constant $C_{\mu}(K)\ge 0$ such that $\mu(K+z)\le C_{\mu}(K)$, $z\in\mathbb{R}^n$.
\begin{lemma}[\cite{G}]\label{l-2-2}
A positive Borel measure $\mu$ on $\mathbb{R}^n$ is translation-bounded if and only if $D^+(\mu)<\infty$.
\end{lemma}
\begin{theorem}[\cite{G}]\label{t-2-2}
Let $f\in L^1(\mathbb{R}^n)$ with $f\ge 0$ and let $\mu$ be a positive Borel measure on $\mathbb{R}^n$. If there exists a constant $C>0$ such that $f*\mu\le C$ a.e. on $\mathbb{R}^n$, then $\int_{\mathbb{R}^n}f(x) \ dx \ D^+(\mu)\le C$. If, in addition, $\mu$ is translation-bounded and there
 exists a constant $C>0$ such that $f*\mu\ge C$ a.e. on $\mathbb{R}^n$, then $\int_{\mathbb{R}^n}f(x) \ dx \ D^-(\mu)\ge C$. 
\end{theorem}

Let $\Lambda\subset\mathbb{R}^n$ be a discrete subset, a measurable set $K\subset\mathbb{R}^n$ is said to $\Lambda$-tile $\mathbb{R}^n$, if $\{K+\lambda\}_{\lambda\in\Lambda}$ is a partition of $\mathbb{R}^n$ up to zero Lebesgue measure sets, or equivalently, 
\begin{eqnarray}\label{e-2-2}
\sum\limits_{\ell\in\Lambda}\chi_{K}(x+\ell)=1 \ {\rm{for}} \ a.e. \ x\in\mathbb{R}^n.
\end{eqnarray}
The tiling property of a measurable set $K\subset\mathbb{R}^n$ gives some information on the Beurling density of $\Lambda$ as shown in Lemma \ref{l-2-1}. 
\begin{lemma}\label{l-2-1}
Let $\Lambda$ be a discrete subset of $\mathbb{R}^n$. If a measurable subset $K\subset\mathbb{R}^n$ $\Lambda$-tiles $\mathbb{R}^n$, 
then the uniform Beurling density $\mathcal{D}(\Lambda)$ of $\Lambda$ exists and $\lvert K\rvert\, \mathcal{D}(\Lambda)=1$.
\end{lemma}
\begin{proof}
Let $\mu:=\sum\limits_{\ell\in\Lambda}\delta_{\ell}$, then $\mu$ defines a positive  Borel measure and we have
\begin{eqnarray}\label{e-2-5}
\sum\limits_{\ell\in\Lambda}\chi_{K+\ell}(x)=\sum\limits_{\ell\in\Lambda}\chi_K*\delta_{\ell}=\chi_K*\sum\limits_{\ell\in\Lambda}\delta_{\ell}=\chi_K*\mu.
\end{eqnarray}
 Using our assumption that $K$ $\Lambda$-tile $\mathbb{R}^n$, we obtain from (\ref{e-2-5}) that $\chi_K*\mu=1$. Theorem \ref{t-2-2} and Lemma \ref{l-2-2} 
then imply that 
$$\int_{\mathbb{R}^n}\chi_K(x) \ dx \ D^+(\mu)=\int_{\mathbb{R}^n}\chi_K(x) \ dx \ D^-(\mu)=1,$$
which yields that $\lvert K\rvert \mathcal{D}(\mu)=1$. This shows that $\mathcal{D}(\Lambda)$ exists 
and proves our claim.
\end{proof}

\begin{theorem}[\cite{G}]\label{t-2-3}
Let $\mu$ be a positive Borel measure on $\mathbb{R}^n$. Then 
$$D^+(\mu)=\inf\limits_{\substack{f\ge0\\\int f=1}}\lVert \mu*f\rVert_{\infty}.$$
\end{theorem}
Using Theorem \ref{t-2-3}, the following property of the upper Beurling density of a discrete set is easily obtained.
\begin{proposition}\label{l-2-0}
Let $\Lambda\subset\mathbb{R}^n$ be a discrete set and let $C\in M_n(\mathbb{R}^n)$ be an invertible matrix. 
Then $$\lvert\det C\rvert \,D^+(C\Lambda)=D^+(\Lambda).$$
\end{proposition}
\begin{proof}
Define $\mu=\sum\limits_{\lambda\in\Lambda}\delta_{\lambda}$ and 
$\widetilde{\mu}=\lvert\det C\rvert\sum\limits_{\lambda\in\Lambda}\delta_{C\lambda}$. 
Then, we have
\begin{eqnarray*}
\int_{\mathbb{R}^n}f(x) \ d\widetilde{\mu}(x)=\int_{\mathbb{R}^n}f(x)\ d\mu(C^{-1}x)=\lvert\det C\rvert\int_{\mathbb{R}^n}f(Cx) \ d\mu(x).
\end{eqnarray*}
Using the definition of $\widetilde{\mu}$ and the previous equality, we obtain
\begin{eqnarray*}
\widetilde{\mu}*f=\lvert\det C\rvert \int_{\mathbb{R}^n}f(C(x-y)) \ d\mu(y)=\mu*h,
\end{eqnarray*}
where $h(x)=\lvert\det C\rvert\, f(Cx)$. In particular, if  $f$ satisfies $f\ge 0$ and $\int_{\mathbb{R}^n}f(x) \ dx=1$, then 
so does $h$ and vice versa. Thus, 
$$\inf\limits_{\substack{f\ge0\\\int f=1}}\lVert \mu*f\rVert_{\infty}
=\inf\limits_{\substack{f\ge0\\\int f=1}}\lVert \widetilde{\mu}*f\rVert_{\infty}.
$$ 
and it follows from Theorem \ref{t-2-3} that $D^+(\mu)=D^+(\widetilde{\mu})$, i.e.~ $\lvert\det C\rvert\, D^+(C\Lambda)=D^+(\Lambda)$.
\end{proof}

If $U$ is a non-empty subset of $\mathbb{R}^n$, we denote the diameter of $U$ as $\text{diam} (U)$, 
which is defined by $\text{diam} (U)=\sup\{\lVert x-y\rVert: x,y\in U\}.$
Next, let us introduce the definition of $s$-dimensional Hausdorff measure of a 
subset $E\subset\mathbb{R}^n$ that we will use in this paper. Let $E$ be a subset of $\mathbb{R}^n$ and 
let $s$ be a non-negative number. For $\delta>0$, define 
\begin{eqnarray*}
\mathcal{H}_{\delta}^s(E)=\inf\Big\{\sum\limits_{i=1}^{\infty}[\text{diam}(U_i)]^s: E\subseteq \bigcup\limits_{i=1}^{\infty}U_i, \text{diam}(U_i)<\delta\Big\}.
\end{eqnarray*}
$s$-dimensional Hausdorff measure of a subset $E\subset\mathbb{R}^n$ is defined by
\begin{eqnarray*}
\mathcal{H}^s(E)=\lim\limits_{\delta\to 0}\mathcal{H}_{\delta}^s(E)=\sup\limits_{\delta>0}\mathcal{H}_{\delta}^s(E).
\end{eqnarray*}
Under this definition, the $n$-dimensional Hausdorff measure of $\mathbb{R}^n$ is related to the usual Lebesgue measure
 if $n$ is a positive integer. Clearly, the definitions of Lebesgue measure and $\mathcal{H}^1$ on $\mathbb{R}$ coincide.
 For $n>1$, they differ only by a constant multiple, i.e. 
if $E\subset\mathbb{R}^n$, then $\lvert E\rvert=c_n\mathcal{H}^n(E)$, where $c_n=\pi^{\frac{1}{2}n}/2^n\Gamma(\frac{n}{2}+1)$ (see \cite{FA}).

It has been showed that for a self-similar set, the OSC is equivalent to $\mathcal{H}^s(K)>0$ in Euclidean space,
 where $s$ is its similarity dimension (see e.g. \cite{BHR, S, S96}). More generally, for a self-affine set, He and Lau proved the following result.
\begin{theorem}[\cite{HL}]\label{t-1-1}
The IFS $\{f_d\}_{d\in\mathcal{D}}$ satisfies the OSC if and only if $\mathcal{D}_{\infty}$ is a uniformly discrete set and
the  $(\text{card}(\mathcal{D}))^k$ expansions in $\mathcal{D}_k$ are distinct for all $k\ge 1$.
\end{theorem}

\section{The Lebesgue measure of self-affine sets}
The Beurling density of discrete sets has been used extensively in the study of Fourier frames (see e.~g.~\cite{L1967, GR1996, C2003, S1995}). 
In this section, we will give some applications of the notion of Beurling density to the theory of self-affine sets.
\begin{theorem}\label{t-2-4}
Let $B\in M_n(\mathbb{R})$ be an expansive matrix and let $\mathcal{D}$ be a finite subset of $\mathbb{R}^n$ with $\text{card}(\mathcal{D})=\lvert\det B\rvert=m\in\mathbb{Z}$. Then, 
$\lvert K(B,\mathcal{D})\rvert=(D^+(\mu))^{-1}$, where $\mu$ is defined by (\ref{e-1-2}), 
with the convention that $\lvert K(B,\mathcal{D})\rvert=0$ if $D^+(\mu)=\infty$. 
\end{theorem}
\begin{proof}
We will consider the two cases, $\lvert K(B,\mathcal{D})\rvert>0$ and $\lvert K(B,\mathcal{D})\rvert=0$, separately.\\
{\bf Case I:} Assume that $\lvert K(B,\mathcal{D})\rvert>0$. Then the set $\mathcal{D}_{\infty}$ 
is uniformly discrete by Theorem \ref{t-2-1}. Hence, by Lemma \ref{l-2-2}, the measure 
$\mu:=\sum\limits_{\lambda\in\mathcal{D}_{\infty}}\delta_{\lambda}$ is translation-bounded and the
sets $K+\ell, \ell\in\mathcal{D}_{\infty}$ are essentially disjoint since $B^k K=\bigcup\limits_{\ell\in\mathcal{D}_k}K+\ell$ 
has Lebesgue measure $m^k\lvert K\rvert$.  Thus, we have 
\begin{eqnarray}\label{e-2-6}
\mu*\chi_K(x)=\sum\limits_{\lambda\in\mathcal{D}_{\infty}}\chi_K(x-\lambda)
=\sum\limits_{\lambda\in\mathcal{D}_{\infty}}\chi_{K+\lambda}(x)=\chi_{\bigcup\limits_{\ell\in\mathcal{D}_{\infty}}(K+\ell)}(x)\le 1.
\end{eqnarray}
It follows from (\ref{e-2-6}) and Theorem \ref{t-2-2} that 
$\int_{\mathbb{R}^n}\chi_K(d) \ dxD^+(\mu)\le 1$, i.e. $D^+(\mu)\le\frac{1}{\lvert K\rvert}$. 
To prove the converse inequality, use the identity 
$D^+(\mu)=\inf\limits_{\substack{f\ge0\\\int f=1}}\lVert \mu*f\rVert_{\infty}$ from Theorem \ref{t-2-3}. 
Thus for any fixed $\varepsilon>0$, there exists $f\ge 0$ with $\int f=1$ such that 
$\lVert \mu*f\rVert_{\infty}\le D^+(\mu)+\varepsilon$, which implies that $\mu*f(x)\le D^+(\mu)+\varepsilon$ a.e. $x\in\mathbb{R}^n$.
 Then, using the definition of convolution, we have
\begin{eqnarray}\label{e-2-7}
\mu*f*\frac{\chi_K}{\lvert K\rvert}\le D^+(\mu)+\varepsilon.
\end{eqnarray}
On the other hand,
\begin{eqnarray*}\label{e-2-8}
&&\mu*f*\frac{\chi_K}{\lvert K\rvert}(x)\nonumber =\mu*\frac{\chi_K}{\lvert K\rvert}*f(x)
=\frac{1}{\lvert K\rvert}\int_{\mathbb{R}^n}\,\big(\mu*\chi_K\big)(x-y)\,f(y) \ dy\nonumber\\
&=&\frac{1}{\lvert K\rvert}\left(\int_{B(0,N)}\,\big(\mu*\chi_K\big)(x-y)\,f(y) \, dy
+\int_{\mathbb{R}^n\setminus B(0,N)}\,\big(\mu*\chi_K\big)(x-y)\,f(y)\,dy\right).
\end{eqnarray*}
Since $\mu*\chi_K\le 1$, $f\ge 0$ and $\int_{\mathbb{R}^n}f(y) \ dy=1$, we have $\lVert\mu\ast f\ast\chi_K\rVert_{\infty}\le1$.
 Furthermore, given $\varepsilon>0$, using the Lebesgue dominated convergence theorem, we have
\begin{eqnarray}\label{e-2-9}
\int_{\mathbb{R}^n\setminus B(0,N)}\,\big( \mu*\chi_K\big)(x-y)\,f(y)\, dy\le
\int_{\mathbb{R}^n\setminus B(0,N)}\,f(y)\, dy <\varepsilon  \ {\rm{if}} \ N\ge N_0.
\end{eqnarray}
The fact that $\lvert K\rvert>0$ implies that $K$ has a non-empty interior (\cite{LWG}). Thus, 
for any given $N>0$, $\bigcup\limits_{m\ge 0}B^mK$ must contain a ball $B(x_N, 2N)$ for some $x_N\in\mathbb{R}^n$. It follows from (\ref{e-2-6}) that
\begin{eqnarray*}
\big(\mu*\chi_K\big)(x-y)=\chi_{\bigcup\limits_{\ell\in\mathcal{D}_{\infty}}(K+\ell)}(x-y)=\chi_{\bigcup\limits_{m\ge 0}B^mK}(x-y).
\end{eqnarray*}
Note that for $y\in B(0, N)$, $\chi_{\bigcup\limits_{m\ge 0}B^mK}(x-y)$ has the value $1$ in the ball $B(x_N,N)$.
 Therefore, for $x\in B(x_N,N)$, we have 
\begin{eqnarray}\label{e-2-10}
\int_{B(0,N)}\,\big(\mu*\chi_K\big)(x-y)\,f(y)\, dy=\int_{B(0,N)}\,f(y) \,dy\rightarrow 1 \ {\rm{as}} \ N\rightarrow\infty.
\end{eqnarray}
We deduce from (\ref{e-2-9}) and (\ref{e-2-10}) that $\lVert\mu*f*\frac{\chi_K}{\lvert K\rvert}\rVert_{\infty}=\frac{1}{\lvert K\rvert}$.
 Therefore, using (\ref{e-2-7}), $\frac{1}{\lvert K\rvert}\le D^+(\mu)+\varepsilon.$ Since $\varepsilon>0$ is arbitrary, $\frac{1}{\lvert K\rvert}\le D^+(\mu)$. This proves our claim if $\lvert K(B,\mathcal{D})\rvert>0$.\\
{\bf Case II:} Assume that $\lvert K(B,\mathcal{D})\rvert =0$. Then, by Theorem \ref{t-2-1},
 either the $m^k$ expansions in $\mathcal{D}_k$ are not distinct for some $k$ or the set $\mathcal{D}_{\infty}$ is not a uniformly discrete set. \\
Let us assume first that the $m^k$ expansions in $\mathcal{D}_k$ are not distinct for a given $k$. There exist then  $a\in\mathcal{D}_k$ 
which can be represented in two different ways in terms of the digits in $\mathcal{D}$, i.e.
\begin{eqnarray*}
a=\sum\limits_{j=0}^{k-1}B^jd_j=\sum\limits_{j=0}^{k-1}B^jd_j^{\prime}, \ d_j,d_j^{\prime}\in\mathcal{D},\,\,\text{with}\,\,
d_j\ne dj^{\prime}\,\,\text{for at least one}\,\,j.
\end{eqnarray*}
Then the element $a+B^ka\in\mathcal{D}_{2k}$ has at least four distinct representations and more generally, 
$\sum\limits_{j=0}^{M-1}B^{kj}a$ has at least $2^M$ distinct expansions in $\mathcal{D}_{Mk}$. It follows that if 
$z_M=\sum\limits_{j=0}^{M-1}B^{kj}a$, then $\mu(\{z_M\})\ge 2^M$. Hence, for any $N>0$, we have that
 $\sup\limits_{z\in\mathbb{R}^n}\frac{\mu(I_N(z))}{N^n}=\infty$ and, in particular,
\begin{eqnarray*}
D^+(\mu)=\limsup\limits_{N\rightarrow\infty}\sup\limits_{z\in\mathbb{R}^n}\frac{\mu(I_N(z))}{N^n}=\infty.
\end{eqnarray*}
Let us assume now that $\mathcal{D}_{\infty}$ is not a uniformly discrete set. Then there exists $k_1\ge 1$ and 
$x_1,y_1\in\mathcal{D}_{k_1}\subset\mathcal{D}_{\infty}$ with $x_1\ne y_1$ such that 
$\lVert x_1-y_1\rVert<\frac{1}{2}$.
Let $ F_1=\{x_1,y_1\}$ and $w_1=x_1$. Then $F_1\subset\mathcal{D}_{k_1}\subset\mathcal{D}_{\infty}$ and 
$\lVert z_1-w_1\rVert<\frac{1}{2}$ for any $z_1\in F_1$. We define $S_1=0$. More generally, if $M\ge 2$ and $k_j, S_j$ and 
$x_j,y_j\in\mathcal{D}_{k_j}$, $F_j\subset\mathcal{D}_{S_j}$ have been defined for $1\le j\le M-1$, we let 
$S_M=\sum\limits_{\ell=1}^{M-1}k_{\ell}$ and choose $k_M$ and $x_M,y_M\in\mathcal{D}_{k_M}\subset\mathcal{D}_{\infty}$ 
 with $x_M\ne y_M$ such that $\lVert x_M-y_M\rVert<\frac{1}{2^M\lVert B\rVert^{S_M}}$. We let
\begin{eqnarray*}
&& F_M=\Big\{z_1+B^{k_1}z_2+\dotsb+B^{S_M}z_M, \ z_i\in\{x_i,y_i\}, \ {\rm{for}} \ 1\le i\le M\Big\},\\
&& w_M=x_1+B^{k_1}x_2+\dotsb+B^{S_M}x_M.
\end{eqnarray*}
Then $F_M\subset\mathcal{D}_{S_M+k_M}\subset\mathcal{D}_{\infty}$, $w_M\in\mathcal{D}_{S_M+k_M}$ and for any $z\in F_M$, we have
\begin{eqnarray*}
\lVert z-w_M\rVert&=&\lVert(z_1-x_1)+B^{k_1}(z_2-x_2)+\dotsb+B^{S_M}(z_M-x_M)\rVert\\
&<&\frac{1}{2}+\lVert B\rVert^{k_1}\frac{1}{4\lVert B\rVert^{k_1}}+\dotsb+\lVert B\rVert^{S_M}\frac{1}{2^M\lVert B\rVert^{S_M}}\\
&=&(\frac{1}{2}+\frac{1}{4}+\dotsb\frac{1}{2^M})< 1.
\end{eqnarray*}
It follows that $\mu(I_2(w_M))\ge 2^M$. Therefore, for any $N\ge 2$,  we have that $\sup\limits_{z\in\mathbb{R}^n}\frac{I_N(z)}{N^n}=\infty$ and $D^+(\mu)=\infty$ as before.  
\end{proof}

Using Theorem \ref{t-2-4} together with Theorem \ref{t-1-1} and Lemma \ref{l-2-2}, we deduce the following result.
\begin{theorem}\label{t-2-5}
If the measure $\mu$ is defined as in (\ref{e-1-2}), then the following conditions are equivalent.
\begin{enumerate}[(i)]
\item The IFS $\{f_i\}_{i=1}^m$ satisfies the open set condition.
\item The $m^k$ expansions in $\mathcal{D}_k$ are distinct for all $k\ge 1$ and $\mathcal{D}_{\infty}$ is a uniformly discrete set.
\item $\lvert K(B,\mathcal{D})\rvert>0$
\item $0<D^+(\mu)<\infty$.
\item $\mu$ is translation-bounded.
\end{enumerate}
\end{theorem}
\begin{proof}
The equivalence of (i) and (ii) follows from Theorem \ref{t-1-1}, that of (ii) and (iii) from Theorem \ref{t-2-1} and that of (iii) and (iv) from Theorem \ref{t-2-4}. The fact that $\mu$ is translation-bounded is equivalent to $D^+(\mu)<\infty$ by Lemma \ref{l-2-2}. In that case, we must have $D^+(\mu)>0$ by Theorem \ref{t-2-4} since $\lvert K(B,\mathcal{D})\rvert<\infty$. This proves the equivalence of (iv) and (v).
\end{proof}

Consider now the example where $B=2$ and $\mathcal{D}=\{0,1\}$ in dimension $1$. Then, $K(B,\mathcal{D})=[0,1]$ and $\mu=\sum\limits_{n=0}^{\infty}\delta_n$ yielding $D^+(\mu)=1$ and $D^-(\mu)=0$. On the other hand, if $B=-2$ and $\mathcal{D}=\{0,1\}$, we have $K(B,\mathcal{D})=[-\frac{2}{3},\frac{1}{3}]$ and $\mu=\sum\limits_{n=-\infty}^{\infty}\delta_n$. Thus, $D^+(\mu)=D^-(\mu)=1$ in that case. We note that $0$ belongs to the boundary of $K(B,\mathcal{D})$ in the first case and in the interior of it in the second case. This fact holds in general, i.e. the Beurling density of $\mu$ can also be used to check whether or not a self-affine tile contains a neighborhood of $0$, or equivalently, to check whether or not it is a $\mathcal{D}_{\infty}$-tiling set as shown in the proof of the following theorem.
\begin{theorem}\label{t-2-6}
Let $K=K(B,\mathcal{D})$ be a self-affine tile. Then 
\begin{enumerate}[(a)]
\item $K$ contains a neighborhood of $0$ if and only if $D^+(\mathcal{D}_{\infty})=D^-(\mathcal{D}_{\infty})=\frac{1}{\lvert K\rvert}$. 
\item $K$ does not contain a neighborhood of $0$ if and only if $D^+(\mathcal{D}_{\infty})=\frac{1}{\lvert K\rvert}$ and $D^-(\mathcal{D}_{\infty})=0$.
\end{enumerate}
\end{theorem}
\begin{proof}
We have $\lvert K\rvert>0$ since $K=K(B,\mathcal{D})$ is a self-affine tile and the sets $K+\ell, \ \ell\in\mathcal{D}_{\infty}$ are essentially disjoint by the proof in Theorem \ref{t-2-4}.
Suppose that $K(B,\mathcal{D})$ contains a neighborhood of $0$. Then $\bigcup\limits_{k\in\mathbb{Z}}B^k K=\bigcup\limits_{\ell\in\mathcal{D}_{\infty}}K+\ell=\mathbb{R}^n$ since $B$ is expansive.
Thus, $K$ is a $\mathcal{D}_{\infty}$-tiling set. It follows from Lemma \ref{l-2-1} that $D^+(\mathcal{D}_{\infty})=D^-(\mathcal{D}_{\infty})=\frac{1}{\lvert K\rvert}$. \\
 Assume that $K$ does not contain a neighborhood of $0$. Then $(K+\mathcal{D}_{\infty})^c$ is a non-empty open set in $\mathbb{R}^n$ since $K$ is not a $\mathcal{D}_{\infty}$-tiling set and $\mathcal{D}_{\infty}$ is a discrete set. Thus we can find a ball $D(a,r)=\Big\{x\in\mathbb{R}^n,\lVert x-a\rVert<r\Big\}$ contained in $(K+\mathcal{D}_{\infty})^c$. On the other hand, since, $B(K+\mathcal{D}_{\infty})=K+\mathcal{D}_{\infty}$, we have $B^m D(a,r)\subset (K+\mathcal{D}_{\infty})^c$ for any $m\ge 0$. Since $B$ is expansive, $B^m D(a,r)$ contains a cube $I_{N_m}(B^ma)$ with $\lim\limits_{m\rightarrow\infty}N_m=\infty$. Hence, we have
\begin{eqnarray*}
D^-(\mathcal{D}_{\infty})=\liminf\limits_{N\rightarrow \infty}\inf\limits_{z\in\mathbb{R}^n}\frac{\text{card}(\mathcal{D}_{\infty}\bigcap (I_N(z))}{N^n}\le \lim
\limits_{m\rightarrow\infty}\frac{\text{card}(\mathcal{D}_{\infty}\bigcap I_{N_m}(B^ma))}{N_m^n}=0.
\end{eqnarray*}
Furthermore, by Theorem \ref{t-2-4}, we have $D^+(\mathcal{D}_{\infty})=\frac{1}{\lvert K\rvert}>0$. This proves our claim.
\end{proof}

As we mentioned before, if $\text{card}(\mathcal{D})>\lvert\det B\rvert$, then the translates  $K+d, d\in\mathcal{D}$, can overlap on a set of
positive measure if $\lvert K\rvert>0$, which makes the computation of the Lebesgue measure of $K$ more difficult.
 The analogue of Theorem \ref{t-2-4} does not hold in this situation as illustrated by the next example.
\begin{example}\label{ex-2-1}
In dimension one, consider the set $K$ associated with the dilation $\frac{3}{2}$ and the digit set $\mathcal{D}=\{0,1\}$, i.e. the set $K$ satisfies that $\frac{3}{2}K=K\bigcup (K+1)$. Then $K=[0,2]$ and $\lvert K\rvert=2$. By the definition of $\mathcal{D}_s$, we have
\begin{eqnarray*}
\mathcal{D}_s=\Big\{\sum\limits_{j=0}^{s-1}B^j\ell_j,\ell_j\in\mathcal{D}\Big\}=
\Big\{\sum\limits_{j=0}^{s-1}\,\left(3/2\right)^j\ell_j,\ell_j\in\{0,1\}\Big\}.
\end{eqnarray*}
The number of elements in $\mathcal{D}_s$ is $2^s$ and the largest element in $\mathcal{D}_s$ is $\sum\limits_{j=0}^{s-1}(\frac{3}{2})^j=2[(\frac{3}{2})^s-1]$. Then, using the definition of $D^+(\mu)$, where $\mu$ is defined by (\ref{e-1-2}), we have
\begin{eqnarray*}
D^+(\mu)\ge\lim\limits_{s\rightarrow\infty}\frac{2^s}{2[(\frac{3}{2})^s-1]}=\infty.
\end{eqnarray*}
This shows that $\lvert K\rvert\ne(D^+(\mu))^{-1}$.
\end{example}

\section{The Hausdorff measure of self-affine sets}

In this section, we will limit our discussion to self-similar sets $K:=K(B,\mathcal{D})$, i.e. 
$B$ will be assumed to be a similarity with scaling factor $\rho>1$ and $\text{card}(\mathcal{D})<\lvert\det B\rvert$. 

Our main goal in this section is to extend the results of section 3 concerning the Lebesgue measure of self-affine set. In this section, the Lebesgue measure will be replaced by the $s$-Hausdorff measure and the Beurling density by an analogous notion of ``$s$-density". 
\begin{definition}\label{d-3-1}
Let $\mu$ be a positive Borel measure on $\mathbb{R}^n$. Define the upper $s$-density of $\mu$ to be the quantity
\begin{eqnarray*}
\mathcal{E}_s^+(\mu)=\limsup\limits_{r\rightarrow\infty}\sup\limits_{\text{diam}(U)\ge r>0}\frac{\mu(U)}{[\text{diam}(U)]^s},
\end{eqnarray*}
where the supremum is over all compact convex sets $U$ with $\text{diam}(U)\ge r>0$.
\end{definition}
We also recall the definition of the convolution of two measures. Let $\mu$ be a Borel measure and let $\sigma$ be a Borel probability measure. The convolution $\mu*\sigma$ is defined by 
\begin{eqnarray*}
\int_{\mathbb{R}^n}\phi(z)d(\mu*\sigma)(z):=\int_{\mathbb{R}^n}\int_{\mathbb{R}^n}\phi(x+y)d\mu(x)d\sigma(y),
\end{eqnarray*}
for any compactly supported continuous function $\phi$ on $\mathbb{R}^n$. If $E$ is a bounded Borel set, we can define $(\mu*\sigma)(E)$ by replacing $\phi$ by $\chi_E$, the characteristic function of $E$, in the previous formula.
\begin{lemma}\label{l-3-1}
Let $\mu$ and $\sigma$ be  positive Borel measures on $\mathbb{R}^n$ with  $\sigma$ being also a  probability measure. Then 
$$\mathcal{E}_s^+(\mu*\sigma)=\mathcal{E}_s^+(\mu).$$
\end{lemma}
\begin{proof}
By the definition of $\mathcal{E}_s^+(\mu)$, we get
\begin{eqnarray}\label{e-3-3}
\mathcal{E}_s^+(\mu*\sigma)&=&\limsup\limits_{r\rightarrow\infty}\sup_{\text{diam}(U)\ge r>0}\frac{\mu*\sigma(U)}{[\text{diam}(U)]^s}\nonumber\\
&=&\limsup\limits_{r\rightarrow\infty}\sup_{\text{diam}(U)\ge r>0}\frac{\int_{\mathbb{R}^n}\int_{\mathbb{R}^n}\chi_{U}(x+y)d\mu(x)d\sigma(y)}{[\text{diam}(U)]^s}\nonumber\\
&=& \limsup\limits_{r\rightarrow\infty}\sup_{\text{diam}(U)\ge r>0}\frac{\int_{\mathbb{R}^n}\mu(U-y)d\sigma(y)}{[\text{diam}(U)]^s},
\end{eqnarray}
where the supremum is over all convex sets $U$ with $\text{diam}(U)\ge r>0$. Since $\sigma$ is a Borel probability measure, using (\ref{e-3-3}), we have
\begin{eqnarray*}
\limsup\limits_{r\rightarrow\infty}\sup_{\text{diam}(U)\ge r>0}\frac{\int_{\mathbb{R}^n}\mu(U-y)d\sigma(y)}{[\text{diam}(U)]^s}&\le& \limsup\limits_{r\rightarrow\infty}\sup_{\text{diam}(U)\ge r>0}\sup\limits_{y\in\mathbb{R}^n}\frac{\mu(U-y)}{[\text{diam}(U)]^s}\\
&=&\limsup\limits_{r\rightarrow\infty}\sup_{\text{diam}(U)\ge r>0}\frac{\mu(U)}{[\text{diam}(U)]^s},
\end{eqnarray*}
which implies that $\mathcal{E}_s^+(\mu*\sigma)\le \mathcal{E}_s^+(\mu)$. For the converse inequality, we can assume that $\mathcal{E}_s^+(\mu*\sigma)<\infty.$  Let $V$ be the convex hull of the sets $U$ and $U+y, \ y\in D(0,R)$ for some fixed $R>0$. Then $U\subseteq V\bigcap (V-y)$ and $ \text{diam}(V)\le \text{diam} (U)+R$. Furthermore, we have
\begin{eqnarray}\label{e-3-4}
\frac{\mu(U)}{[\text{diam} (U)]^s}\le\frac{\mu(V-y)}{[\text{diam} (U)]^s}.
\end{eqnarray}
It follows from (\ref{e-3-4}) that, for fixed $R>0$,
\begin{eqnarray*}
\frac{\int_{D(0,R)}\mu(U)d\sigma(y)}{[\text{diam} (U)]^s}
&\le&\frac{\int_{D(0,R)}\mu(V-y)d\sigma(y)}{[\text{diam} (V)]^s}\cdot\frac{[\text{diam} (V)]^s}{[\text{diam} (U)]^s}\\
&\le&\frac{\int_{D(0,R)}\mu(V-y)d\sigma(y)}{[\text{diam} (V)]^s}\cdot\frac{(\text{diam} (U)+R)^s}{[\text{diam} (U)]^s}.
\end{eqnarray*}
Thus we have
\begin{align*}
&\limsup\limits_{r\rightarrow\infty}\sup\limits_{\text{diam} (U)\ge r>0}\frac{\int_{D(0,R)}\mu(U)d\sigma(y)}{[\text{diam} (U)]^s}\\
\le&
\limsup\limits_{r\rightarrow\infty}\sup\limits_{\text{diam} (V)\ge r>0}\frac{\int_{D(0,R)}\mu(V-y)d\sigma(y)}{[\text{diam} (V)]^s}\\
\le&
\limsup\limits_{r\rightarrow\infty}\sup\limits_{\text{diam} (V)\ge r>0}\frac{\int_{\mathbb{R}^n}\mu(V-y)d\sigma(y)}{[\text{diam} (V)]^s}.
\end{align*}
Letting $R\rightarrow\infty$, we obtain that $\mathcal{E}_s^+(\mu)\le \mathcal{E}_s^+(\mu*\sigma)$, which yields the converse inequality.
\end{proof}

It is well-known (\cite{H}) that the IFS $\{f_d\}_{d\in\mathcal{D}}$ determines a unique Borel probability measure $\sigma$ supported on the set $K(B,\mathcal{D})$ satisfying 
\begin{eqnarray}\label{e-3-2}
\int f \ d\sigma=\frac{1}{\text{card}(\mathcal{D})}\sum\limits_{d\in\mathcal{D}}\int f\circ f_d \ d\sigma,
\end{eqnarray}
for all compactly supported continuous function $f$ on $\mathbb{R}^n$.   

\begin{lemma}\label{l-3-2}
Let $m=\text{card}(\mathcal{D})$ and let $\sigma$ be the Borel probability measure supported on $K(B,\mathcal{D})$ which satisfies (\ref{e-3-2}).
 Define $\mu_N=\sum\limits_{d_0,\dotsc,d_{N-1}\in\mathcal{D}}\delta_{d_0+Bd_1+\dotsb+B^{N-1}d_{N-1}}$. Then,
 for any Borel measurable set $W\subset\mathbb{R}^n$, we have $\sigma(B^{-N}W)=\frac{1}{m^N}\,\left(\mu_N*\sigma\right)(W)$.
\end{lemma}
\begin{proof}
For any Borel measurable set $W\subset\mathbb{R}^n$, we deduce from the identity (\ref{e-3-2}) that
\begin{eqnarray*}
\sigma(B^{-N}W)
&=&\int_{\mathbb{R}^n}\chi_{B^{-N}W}(x) \ d\sigma(x)\\
&=&\frac{1}{m^N}\sum\limits_{d_1,d_2,\dotsc,d_N\in\mathcal{D}}\int_{\mathbb{R}^n}\chi_{B^{-N}W}(B^{-N}x+B^{-1}d_1+\dotsm+B^{-N}d_N) \ d\sigma(x)\\
&=&\frac{1}{m^N}\sum\limits_{d_1,d_2,\dotsc,d_N\in\mathcal{D}}\int_{\mathbb{R}^n}\chi_{W}(x+B^{N-1}d_1+\dotsb+d_N) \ d\sigma(x)\\
&=&\frac{1}{m^{N}}\int_{\mathbb{R}^n}\chi_W(x) \ d(\sigma*\mu_N)(x)\\
&=&\frac{1}{m^N}\sigma*\mu_N(W).
\end{eqnarray*}
\end{proof}

A subset $E\subset\mathbb{R}^n$ is called an $s$-set ($0\le s\le n$) if 
$E$ is $\mathcal{H}^s$-measurable and $0<\mathcal{H}^s(E)<\infty$. 
The upper convex density of an $s$-set $E$ at $x$ \cite{FA} is defined as
\begin{eqnarray*}
D_c^s(E,x)=\overline{\lim\limits_{r\to 0}}\sup_{0<\text{diam} (U)\le r}\frac{\mathcal{H}^s(E\bigcap U)}{[\text{diam} (U)]^s},
\end{eqnarray*}
where the supremum is over all convex sets $U$ with $x\in U$ and $0<\text{diam} (U)\le r$.
 Note that the upper convex density of an $s$-set $E$ at $x$ can also be defined by 
 $$D_c^s(E,x)=\lim\limits_{r\to 0}\sup_{0<\text{diam} (U)\le r}\frac{\mathcal{H}^s(E\bigcap U)}{[\text{diam} (U)]^s},$$ 
since $\sup\limits_{0<\text{diam} (U)\le r}\frac{\mathcal{H}^s(E\bigcap U)}{[\text{diam} (U)]^s}$ is
 decreasing with respect to $r$. The following theorem will be useful later on in this section.
\begin{theorem}[\cite{FA}]\label{t-3-2}
If $E$ is an $s$-set in $\mathbb{R}^n$, then $D_c^s(E,x)=1$ at $\mathcal{H}^s$-almost all $x\in E$ and $D_c^s(E,x)=0$ at $\mathcal{H}^s$-almost all $x\in E^c$.
\end{theorem}

Theorem \ref{t-3-1} implies that if the IFS $\{f_d\}_{d\in\mathcal{D}}$ satisfies the OSC, 
then the corresponding self-similar set $K:=K(B,\mathcal{D})$ is an $s$-set, 
where $s=\dim_H K=\log_{\rho}^{\text{card}(\mathcal{D})}$ is the Hausdorff dimension (similarity dimension) of $K$. 
In this case, it has been shown in \cite{H} that the probability measure $\sigma$ in (\ref{e-3-2}) 
is a multiple of the restriction of the $s$-Hausdorff measure $\mathcal{H}^s$ to the set $K$, i.e.
\begin{eqnarray}\label{e-3-7}
\sigma=(\mathcal{H}^s(K))^{-1}\mathcal{H}^s\restriction K.
\end{eqnarray}

Combining the formula (\ref{e-3-7}) and Theorem \ref{t-3-2}, we obtain the following corollary.
\begin{corollary}\label{c-3-1}
Let $K:=K(B,\mathcal{D})$ be a self-similar set and let the IFS $\{f_d\}_{d\in\mathcal{D}}$ satisfy the OSC. 
Then
$$
\overline{\lim\limits_{r\to 0}}\sup\limits_{0<\text{diam} (U)\le r}\frac{\sigma(U)}{[\text{diam} (U)]^s}
=(\mathcal{H}^s(K))^{-1},
$$
 where $s$ is the Hausdorff dimension of the set $K$, $\sigma$ is defined by (\ref{e-3-2}) and the supremum is over all convex sets $U$ with $U\bigcap K\ne\emptyset$ and $0<\text{diam} (U)\le r$.
\end{corollary}
\begin{proof}
By assumption, $K$ is an $s$-set. It follows from (\ref{e-3-7}) that 
\begin{eqnarray}\label{e-3-8}
\overline{\lim\limits_{r\to 0}}\sup_{0<\text{diam} (U)\le r}\frac{\sigma(U)}{[\text{diam} (U)]^s}
&=&(\mathcal{H}^s(K))^{-1}\overline{\lim\limits_{r\to 0}}\sup_{0<\text{diam} (U)\le r}\frac{\mathcal{H}^s(K\bigcap U)}{[\text{diam} (U)]^s}\nonumber\\
&=&(\mathcal{H}^s(K))^{-1}\sup\limits_{x\in K}D_c^s(K,x).
\end{eqnarray}
 Since for any $x\in K\bigcap U$, Theorem \ref{t-3-2} implies that $D_c^s(K,x)=1$, we deduce from (\ref{e-3-8}) that 
$$\overline{\lim\limits_{r\to 0}}\sup\limits_{0<\text{diam} (U)\le r}\frac{\sigma(U)}{[\text{diam} (U)]^s}=(\mathcal{H}^s(K))^{-1}.$$
\end{proof}

If $E$ is a self-similar s-set, then the upper convex density of $E$ at $x$ can also be computed as
\begin{eqnarray*}
D_c^s(E,x)=\sup\limits_{r> 0}\sup_{0<\text{diam} (U)\le r}\frac{\mathcal{H}^s(E\bigcap U)}{[\text{diam} (U)]^s},
\end{eqnarray*}
where the supremum is over all convex sets $U$ with $x\in U$ and $0<\text{diam} (U)\le r$. The following lemma clarifies this fact.
\begin{lemma}\label{l-3-3}
Let $K:=K(B,\mathcal{D})$ be a self-similar set, where $B$ is a similarity matrix $B$
 with scaling factor $\rho>1$. Then 
$$
\overline{\lim\limits_{r\to 0}}\sup\limits_{0<\text{diam} (U)\le r}\frac{\mathcal{H}^s(K\bigcap U)}{[\text{diam} (U)]^s}
=\sup\limits_{r> 0}\sup\limits_{0<\text{diam} (U)\le r}
\frac{\mathcal{H}^s(K\bigcap U)}{[\text{diam} (U)]^s},
$$
where the supremum is over all convex sets $U$ with $U\bigcap K\ne \emptyset$ and $0<\text{diam} (U)\le r$.
\end{lemma}
\begin{proof}
Obviously, 
$$
\overline{\lim\limits_{r\to 0}}\sup\limits_{0<\text{diam} (U)\le r}
\frac{\mathcal{H}^s(K\bigcap U)}{[\text{diam} (U)]^s}
\le\sup\limits_{r> 0}\sup\limits_{0<\text{diam} (U)\le r}
\frac{\mathcal{H}^s(K\bigcap U)}{[\text{diam} (U)]^s}.
$$
Conversely, $\mathcal{H}^s(B^{-1}K)=\rho^{-s}\mathcal{H}^s(K)$ since $B$ is a similarity with scaling factor $\rho>1$ (\cite{F, FA}) and we have
\begin{align*}
\sup\limits_{r> 0}\sup\limits_{0<\text{diam} (U)\le r}\frac{\mathcal{H}^s(K\bigcap U)}{[\text{diam} (U)]^s}
&=\sup\limits_{r> 0}\sup\limits_{0<\text{diam} (U)\le r}\frac{\mathcal{H}^s(B^{-1}(K\bigcap U))}{[\text{diam} (B^{-1}U)]^s}\\
&\le\sup\limits_{r> 0}\sup\limits_{0<\text{diam} (U)\le r}\frac{\mathcal{H}^s(K\bigcap B^{-1}U)}{[\text{diam} (B^{-1}U)]^s} \ ({\rm{since}}\,\, K\subset BK)\\
&\le\sup\limits_{r> 0}\sup\limits_{0<\text{diam} (U)\le r}\frac{\mathcal{H}^s(K\bigcap B^{-m}U)}{[\text{diam} (B^{-m}U)]^s}, \ {\rm{for \ any \ m\ge1}}\\
&\le \limsup\limits_{r\to 0}\sup\limits_{0<\text{diam} (U)\le r}\frac{\mathcal{H}^s(K\bigcap U)}{[\text{diam} (U)]^s}.
\end{align*}
This proves our claim.
\end{proof}

 Let $\sigma$ be defined by (\ref{e-3-2}). If the IFS $\{f_d\}_{d\in\mathcal{D}}$ satisfy the OSC, then, using Lemma \ref{l-3-3}, we have
$$\overline{\lim\limits_{r\to 0}}\sup\limits_{0<\text{diam} (U)\le r}\frac{\sigma(U)}{[\text{diam} (U)]^s}=\sup\limits_{r> 0}\sup\limits_{0<\text{diam} (U)\le r}\frac{\sigma(U)}{[\text{diam} (U)]^s},$$ where the supremum is over all convex sets $U$ with $U\bigcap K\ne \emptyset$ and $0<\text{diam} (U)\le r$.

Schief \cite{S} proved that the IFS $\{f_d\}_{d\in\mathcal{D}}$ satisfy the OSC if and only if $\mathcal{H}^s(K)>0$, where $s=\log_{\rho}^{\text{card}(\mathcal{D})}$ is the similarity dimension, in Euclidean space. However, this is no longer the case in general complete metric spaces. Combining the result provided by Schief and Theorem \ref{t-2-5}, we have the following representation for the Hausdorff measure of self-similar sets.
\begin{theorem}\label{t-3-3}
Let $K:=(B,\mathcal{D})$ be a self-similar set and let $s:=\log_{\rho}^{\text{card}(\mathcal{D})}\le n$ be the similarity dimension of $K$. Then $\mathcal{H}^s(K)=(\mathcal{E}_s^+(\mu))^{-1}$, where $\mu$ is defined by (\ref{e-1-2}).
\end{theorem}
\begin{proof} Let us assume first that $\mathcal{H}^s(K)>0$ and thus that the OSC holds (by \cite{S}). By Corollary \ref{c-3-1}, it is enough to prove that 
$$\limsup\limits_{r\rightarrow 0}\sup\limits_{0<\text{diam} (U)\le r}\frac{\sigma(U)}{[\text{diam} (U)]^s}=\mathcal{E}_s^+(\mu),$$
 where the supremum is over all convex sets $U$ with $U\bigcap K\ne\emptyset$ and $0<\text{diam} (U)\le r$. 
It follows from Lemma \ref{l-3-3} that 
 $$\limsup\limits_{r\rightarrow 0}\sup\limits_{0<\text{diam} (U)\le r}\frac{\sigma(U)}{[\text{diam} (U)]^s}=\sup\limits_{r>0}\sup\limits_{0<\text{diam} (U)\le r}\frac{\sigma(U)}{[\text{diam} (U)]^s}$$ and both quantities are thus finite by Corollary \ref{c-3-1}. 
Then, for any given $\varepsilon>0$, there exists a convex set $U_0$ with $U_0\bigcap K\ne\emptyset$ such that
\begin{eqnarray}\label{e-3-9}
\frac{\sigma(U_0)}{[\text{diam} (U_0)]^s}\ge \sup\limits_{r>0}\sup\limits_{0<\text{diam} (U)\le r}\frac{\sigma(U)}{[\text{diam} (U)]^s}-\varepsilon.
\end{eqnarray}
Define $\mu_N=\sum\limits_{d_0,\dotsc,d_{N-1}\in\mathcal{D}}\delta_{d_0+Bd_1+\dotsb+B^{N-1}d_{N-1}}$.
Using Lemma \ref{l-3-2} and Lemma \ref{l-3-1}, we have
\begin{eqnarray}\label{e-3-10}
\frac{\sigma(U_0)}{[\text{diam} (U_0)]^s}=\frac{\sigma*\mu_N(B^NU_0)}{[\text{diam} (B^NU_0)]^s}&\le&\limsup_{r\to \infty}\sup_{\text{diam} (U)\ge r>0}\frac{\sigma*\mu(U)}{[\text{diam} (U)]^s}\nonumber\\
&=&\mathcal{E}_s^+(\sigma*\mu)=\mathcal{E}_s^+(\mu).
\end{eqnarray}
It follows from (\ref{e-3-9}) and (\ref{e-3-10}) that 
$$\limsup\limits_{r\rightarrow 0}\sup\limits_{0<\text{diam} (U)\le r}\frac{\sigma(U)}{[\text{diam} (U)]^s}\le \mathcal{E}_s^+(\mu).$$ 
For any given convex set $U$, using Lemma \ref{l-3-2}, we have,
\begin{eqnarray*}\label{e-3-12}
\frac{\sigma*\mu(U)}{[\text{diam} (U)]^s}&=&\lim\limits_{N\to\infty}\frac{\sigma*\mu_N(U)}{[\text{diam} (U)]^s}=\lim\limits_{N\to\infty}\frac{\sigma(B^{-N}U)}{[\text{diam} (B^{-N} U)]^s}\\
&\le&\limsup\limits_{r\to 0}\sup_{\substack {0<\text{diam} (V)\le r\\ V {\rm{convex}}}}\frac{\sigma(V)}{[\text{diam} (V)]^s}.
\end{eqnarray*}
Using Lemma \ref{l-3-1} again, we have thus that 
$$\mathcal{E}_s^+(\mu)=\mathcal{E}_s^+(\mu*\sigma)\le \limsup\limits_{r\to 0}\sup\limits_{0<\text{diam} (U)\le r}\frac{\sigma(U)}{[\text{diam} (U)]^s}.$$
On the other hand, if $\mathcal{H}^s(K)=0$, then the IFS $\{f_d\}_{d\in\mathcal{D}}$ does not satisfy the OSC by Schief's result \cite{S}. Thus by Theorem \ref{t-2-5}, we have $D^+(\mu)=\infty$. Since $s\le n$ by our assumption, we obtain
\begin{eqnarray*}
D^+(\mu)&=&\limsup\limits_{N\to\infty}\sup\limits_{z\in\mathbb{R}^n}\frac{\mu(I_N(z))}{N^n}\le \limsup\limits_{N\to\infty}\sup\limits_{z\in\mathbb{R}^n}\frac{\mu(I_N(z))}{N^s}\\
&=&\limsup\limits_{N\to\infty}\sup\limits_{z\in\mathbb{R}^n}\frac{\mu(I_N(z))}{[\frac{\text{diam} (I_N(z))}{\sqrt{n}}]^s}\\
&\le& \sqrt{n}^s\limsup\limits_{r\to\infty}\sup\limits_{\text{diam} (U)\ge r>0}\frac{\mu(U)}{[\text{diam} (U)]^s}=\sqrt{n}^s\mathcal{E}_s^+(\mu),
\end{eqnarray*} 
which implies that $\mathcal{E}_s^+(\mu)=\infty$. Therefore, we have $\mathcal{H}^s(K)=(\mathcal{E}_s^+(\mu))^{-1}$.
\end{proof}
\begin{corollary}\label{r-3-1}
Let $K:=K(B,\mathcal{D})$ be a self-similar set and let $s:=\log_{\rho}^{\text{card}(\mathcal{D})}\le n$ be the similarity dimension of $K$. We have $D^+(\mu)=\infty$ if and only if $\mathcal{E}_s^+(\mu)=\infty$, where $\mu$ is defined by (\ref{e-1-2}). 
\end{corollary}
\begin{proof}
By the proof of the previous theorem, the condition $D^+(\mu)=\infty$ implies that $\mathcal{E}_s^+(\mu)=\infty$. On the other hand, if $\mathcal{E}_s^+(\mu)=\infty$, then $\mathcal{H}^s(K)=0$ by Theorem \ref{t-3-3}. Thus the IFS $\{f_d\}_{d\in\mathcal{D}}$ does not satisfy the OSC by Schief's result \cite{S}, which is equivalent to $D^+(\mu)=\infty$ by Theorem \ref{t-2-5}.
\end{proof}

\section{The Hausdorff measure of a class of Cantor sets}

We conclude this paper by providing an example showing how to use Theorem \ref{hausdorff} to compute the Hausdorff measure of a class of Cantor sets which satisfy $NK=K\bigcup(K+d)$, where $N\ge 3$ and $0<d\in\mathbb{R}$. We need first to introduce two lemmas.
\begin{lemma}\label{l-3-5}
Let $K$ be a self-similar set associated with the dilation $N\ge 3$ and the digit set $\mathcal{D}=\{0,d\}\subset\mathbb{R}$ with $0< d\in\mathbb{R}$. If $b=\sum\limits_{j=0}^mN^jr_j$, where $r_j\in\{0,d\}$, then the number of elements in $\mathcal{D}_{\infty}\bigcap[0,b]$ is equal to $\sum\limits_{j=0}^m2^j\frac{r_j}{d}+1$.
\end{lemma}
\begin{proof} We use induction on $m$. If $m=0$, $b=0$ or $d$ and $\mathcal{D}_{\infty}\bigcap[0,b]$ has $1$ or $2$ elements depending on the case, and our claim follows. If the statement is true for $m-1$, where $m\ge 1$, let $b=\sum\limits_{j=0}^mN^jr_j$. If $r_m=0$, our claim follows from our induction hypothesis. If $r_m=d$, consider a number $\sum\limits_{j=0}^mN^js_j$ with $s_j\in\{0,d\}$ belonging to $[0,b]$. If $s_m=0$, there are $2^m$ such numbers in $[0,b]$. If $s_m=d$, then $\sum\limits_{j=0}^mN^js_j\le \sum\limits_{j=0}^mN^jr_j$ is equivalent to $\sum\limits_{j=0}^{m-1}N^js_j\le \sum\limits_{j=0}^{m-1}N^jr_j$. Using our induction hypothesis, the number of elements of $\mathcal{D}_{\infty}$ satisfying the previous inequality is
$\sum\limits_{j=0}^{m-1}2^j\frac{r_j}{d}+1$. Thus the total number of elements in $\mathcal{D}_{\infty}$ less than or equal to $b$ is 
\begin{eqnarray*}
\sum\limits_{j=0}^{m-1}2^j\frac{r_j}{d}+1+2^m=\sum\limits_{j=0}^m2^j\frac{r_j}{d}+1.
\end{eqnarray*}
\end{proof}
\begin{lemma}\label{l-3-4}
Let $K$ be a self-similar set associated with the dilation $N\ge 3$ and the digit set $\mathcal{D}=\{0,d\}\subset\mathbb{R}$ with $0< d\in\mathbb{R}$. Then $\mu([a,b])\le \mu([0,b-a])$, where $[a,b]$ denotes a closed interval in $\mathbb{R}$ and $\mu$ is defined by (\ref{e-1-2}).
\end{lemma}
\begin{proof}
Let $\mu_k=\sum\limits_{d_0,\dotsc,d_{k-1}\in\mathcal{D}}\delta_{d_0+Bd_1+\dotsb+B^{k-1}d_{k-1}}$. Then $\mu=\lim\limits_{k\to\infty}\mu_k$ and by the definition of $\mu_k$, we have $\mu_1=\delta_0+\delta_{d}$ and 
\begin{eqnarray}\label{e-3-16}
\mu_k&=&(\delta_0+\delta_{d})\ast(\delta_0+\delta_{Nd})\ast\dotsb(\delta_0+\delta_{N^{k-1}d})\nonumber\\
&=&\mu_{k-1}\ast(\delta_0+\delta_{N^{k-1}d}) \ {\rm{for}} \ k\ge 2.
\end{eqnarray}
Given any closed interval $[a,b]\subset\mathbb{R}$, in order to prove that $\mu([a,b])\le \mu([0,b-a])$, it is enough to prove that $\mu_k([a,b])\le \mu_k([0,b-a])$ for any $k\ge 1$, which we will do next using induction on $k$. Note that $\text{supp}(\mu_k)=[0,\ell_k]$, where $\ell_k=\sum\limits_{j=0}^{k-1}N^jd$ and $\ell_{k+1}=\ell_k+N^kd$.  We can assume that $0\le a\le b$ since $\text{supp}(\mu_k)\subset[0,\infty)$. It is easy to see that $\mu_1([a,b])\le \mu_1([0,b-a])$. Assume the claim is true for $k$. It follows from (\ref{e-3-16}) that 
\begin{eqnarray*}
\mu_{k+1}=\mu_k\ast(\delta_0+\delta_{N^kd})=\mu_k+\mu_k\ast\delta_{N^kd}.
\end{eqnarray*}
 In the following, we will divide our proof into eight cases.\\
{\bf Case 1:} $0\le a\le b\le \ell_k$. In this case, we have
\begin{eqnarray*}
\mu_{k+1}([a,b])=\mu_k([a,b])\le\mu_k([0,b-a])=\mu_{k+1}([0,b-a]).
\end{eqnarray*}
{\bf Case 2:} $N^kd\le a\le b\le \ell_{k+1}$. In this case, we have
\begin{eqnarray*}
\mu_{k+1}([a,b])=\mu_k([a-N^kd,b-N^kd])\le\mu_k([0,b-a])=\mu_{k+1}([0,b-a]).
\end{eqnarray*}
{\bf Case 3:} $0\le a\le \ell_k$, $N^kd\le b\le \ell_{k+1}$ and $b-N^kd<a$. In this case, $b-a>N^kd-\ell_k>\ell_k$ since $N\ge 3$ and we have
\begin{eqnarray*}
\mu_{k+1}([a,b])&=&\mu_k([a,\ell_k])+\mu_k([0,b-N^kd])\le\mu_k([0,\ell_k])\\
&=&\mu_{k+1}([0,\ell_k])\le\mu_{k+1}([0,b-a]).
\end{eqnarray*}
{\bf Case 4:} $0\le a\le \ell_k$, $N^kd\le b\le \ell_{k+1}$ and $b-N^kd=a$. In this case, $b-a=N^kd>\ell_k$ and we have
\begin{eqnarray*}
\mu_{k+1}([a,b])&=&\mu_k([a,\ell_k])+\mu_k([0,b-N^kd])=\mu_k([a,\ell_k])+\mu_k([0,a])\\
&=&\mu_k([0,\ell_k])+\mu_k(\{a\}).
\end{eqnarray*}
On the other hand, 
\begin{eqnarray*}
\mu_{k+1}([0,b-a])=\mu_{k+1}([0,N^kd])=\mu_k([0,\ell_k])+1\ge\mu_k([0,\ell_k])+\mu_k(\{a\}).
\end{eqnarray*}
{\bf Case 5:} $0\le a\le \ell_k$, $N^kd\le b\le \ell_{k+1}$ and $b-N^kd>a$. In this case, 
\begin{eqnarray*}
\mu_{k+1}([a,b])&=&\mu_{k+1}([a,\ell_k])+\mu_{k+1}([N^kd,a+N^kd))+\mu_{k+1}([a+N^kd,b])\\
&=&\mu_k([a,\ell_k])+\mu_k([0,a))+\mu_k([a,b-N^kd])\\
&=&\mu_k([0,\ell_k])+\mu_k([a,b-N^kd])\\
&\le&\mu_k([0,\ell_k])+\mu_k([0,b-a-N^kd])\\
&=&\mu_{k+1}([0,\ell_k])+\mu_{k+1}([N^kd,b-a])=\mu_{k+1}([0,b-a]).
\end{eqnarray*}
{\bf Case 6:} $\ell_k\le a\le b\le N^kd$. In this case, 
\begin{eqnarray*}
\mu_{k+1}([a,b])=0\le\mu_{k+1}([0,b-a]).
\end{eqnarray*}
{\bf Case 7:} $0\le a\le \ell_k<b<N^kd$. In this case, we have
\begin{eqnarray*}
\mu_{k+1}([a,b])&=&\mu_k([a,\ell_k])\le \mu_k([0,\ell_k-a])=\mu_{k+1}([0,\ell_k-a])\\
&\le&\mu_{k+1}([0,b-a]).
\end{eqnarray*}
{\bf Case 8:} $\ell_k< a<N^kd<b\le \ell_{k+1}$. In this case, we have
\begin{eqnarray*}
\mu_{k+1}([a,b])&=&\mu_{k+1}([N^kd,b])=\mu_k([0,b-N^kd])=\mu_{k+1}([0,b-N^kd])\\
&\le&\mu_{k+1}([0,b-a]).
\end{eqnarray*}
This proves our claim. 
\end{proof}
\begin{proposition}\label{ex-3-1}
Assume that $K$ satisfies that $NK=K\bigcup (K+d)$, where $0<d\in\mathbb{R}$ and $ N\ge 3$. Then $\mathcal{H}^s(K)=(\frac{N-1}{d})^{-s}$, where $s=\log_{N}^2$ is the similarity dimension of $K$.
\end{proposition}
\begin{proof}
For this set $K$, the corresponding digit set $\mathcal{D}=\{0,d\}$ and the similarity dimension $s=\log_{N}^2$, which is also the Hausdorff dimension of $K$ since the IFS $\{f_d\}_{d\in\mathcal{D}}$ satisfies the OSC. 
Considering a sequence of convex sets $U_m=[0, \sum\limits_{j=0}^{m-1}N^j\cdot d]=[0, \frac{N^m-1}{N-1}d]$ and the definition of $\mathcal{E}_s^+(\mu)$, where $\mu$ is defined by (\ref{e-1-2}), we have
\begin{eqnarray}\label{e-3-13}
\mathcal{E}_s^+(\mu)&=&\limsup\limits_{r\to\infty}\sup\limits_{\text{diam} (U)\ge r>0}\frac{\mu(U)}{[\text{diam} (U)]^s}\nonumber\\
&\ge& \lim\limits_{m\to\infty}\frac{\mu(U_m)}{[\text{diam} (U_m)]^s}=\lim\limits_{m\to\infty}\frac{2^m}{(\frac{N^m-1}{N-1}d)^s}=(\frac{N-1}{d})^s.
\end{eqnarray}
Since, in dimension one any convex set is an interval, we can let $U=[a,b], \ a,b\in\mathbb{R}$ and assume that $\ell=\sum\limits_{j=0}^{m_1}N^jr_j$ and $s=\sum\limits_{j=0}^{m_1}N^jd_j$ with $r_j,d_j\in\mathcal{D}$ are the largest and smallest elements of  $\mathcal{D}_{\infty}$ which belong to $[a,b]$ with $r_{m_1}\ne 0$. Then we have $a\le \sum\limits_{j=0}^{m_1}N^jd_j \le \sum\limits_{j=0}^{m_1}N^jr_j\le b$. To obtain an upper-bound for $\mathcal{E}_s^{+}(\mu)$, we can assume, without loss of generality, that $[a,b]=[\sum\limits_{j=0}^{m_1}N^jd_j,\sum\limits_{j=0}^{m_1}N^jr_j]$. Furthermore, $\mu([0,b-a])\ge \mu([a,b])$ by Lemma \ref{l-3-4}, i.e. for intervals having the same length, the one having the maximal number of elements in $\mathcal{D}_{\infty}$ has $0$ as its left boundary point. Hence, we can assume that $a=0$ and $b=\sum\limits_{j=0}^{m}N^jr_j$ with $r_m=d$ and $r_j\in\mathcal{D}$ for $j=1,\dotsc,m-1$. Using Lemma \ref{l-3-5}, we have
\begin{eqnarray*}
\mu([0,b])=\mu([0,\sum\limits_{j=0}^{m}N^jr_j])=\sum\limits_{j=0}^{m}2^j\frac{r_j}{d}+1.
\end{eqnarray*}
On the other hand, if $\sum\limits_{j=0}^{m-1}N^j\cdot d< b\le \sum\limits_{j=0}^{m}N^j\cdot d$, then $b=\sum\limits_{j=0}^{m}N^jr_j$ with $r_m=d$ and $r_j\in\mathcal{D}$ for $0\le j<m$ and we have
\begin{eqnarray}\label{e-3-14}
\frac{\mu([0,b])}{b^s}\le\frac{\mu([0,\sum\limits_{j=0}^{m}N^j\cdot d])}{(\sum\limits_{j=0}^{m}N^j\cdot d)^s}&\Longleftrightarrow& \frac{\sum\limits_{j=0}^{m}2^j\frac{r_j}{d}+1}{(\sum\limits_{j=0}^{m}N^jr_j)^s}\le \frac{2^{m+1}}{(\sum\limits_{j=0}^{m}N^j\cdot d)^s}\nonumber\\
&\Longleftrightarrow &\frac{(\sum\limits_{j=0}^{m}N^j\cdot d)^s}{(\sum\limits_{j=0}^{m}N^jr_j)^s}\le \frac{2^{m+1}}{\sum\limits_{j=0}^{m}2^{j}\frac{r_j}{d}+1}.
\end{eqnarray}
Next, we will prove (\ref{e-3-14}). Since $\sum\limits_{j=0}^{m}N^jr_j\le \sum\limits_{j=0}^{m}N^j\cdot d$ and $0<s<1$, we have the inequality
$$\frac{(\sum\limits_{j=0}^{m}N^j\cdot d)^s}{(\sum\limits_{j=0}^{m}N^jr_j)^s}<\frac{\sum\limits_{j=0}^{m}N^j\cdot d}{\sum\limits_{j=0}^{m}N^jr_j}$$
 and thus, the inequality 
\begin{eqnarray}\label{e-3-17}
\frac{\sum\limits_{j=0}^{m}N^j\cdot d}{\sum\limits_{j=0}^{m}N^jr_j}\le \frac{2^{m+1}}{\sum\limits_{j=0}^{m}2^{j}\frac{r_j}{d}+1}
\end{eqnarray}
would imply (\ref{e-3-14}). Moreover, 
\begin{eqnarray*}
\frac{\sum\limits_{j=0}^{m}N^j\cdot d}{\sum\limits_{j=0}^{m}N^jr_j}\le \frac{2^{m+1}}{\sum\limits_{j=0}^{m}2^{j}\frac{r_j}{d}+1}&\Longleftrightarrow&
 \frac{\frac{N^{m+1}-1}{N-1}d}{\sum\limits_{j=0}^{m}N^jr_j}\le \frac{2^{m+1}}{\sum\limits_{j=0}^{m}2^{j}\frac{r_j}{d}+1}\\
&\Longleftrightarrow&(\sum\limits_{j=0}^{m}2^{j}\frac{r_j}{d}+1)(\frac{N^{m+1}-1}{N-1}d)\le \sum\limits_{j=0}^{m}2^{m+1}N^jr_j,
\end{eqnarray*}
which is equivalent to 
$$\sum\limits_{j=0}^{m-1}(2^{j}\frac{N^{m+1}-1}{N-1}-2^{m+1}N^j)r_j\le \frac{N^m2^{m}(N-2)+2^m-(N^{m+1}-1)}{N-1}d$$
the last inequality is obtained using $r_m=d$. Note that $2^{j}\frac{N^{m+1}-1}{N-1}-2^{m+1}N^j\ge 0$ for any $0\le j\le m-1$ since $N\ge 3$ and 
\begin{eqnarray*}
\sum\limits_{j=0}^{m-1}(2^{j}\frac{N^{m+1}-1}{N-1}-2^{m+1}N^j)d=\frac{N^m2^{m}(N-2)+2^m-(N^{m+1}-1)}{N-1}d.
\end{eqnarray*}
(\ref{e-3-17}) follows and thus, (\ref{e-3-14}) holds. Then by the definition of $\mathcal{E}_s^+(\mu)$ and the above argument, we have
\begin{eqnarray}\label{e-3-15}
\mathcal{E}_s^+(\mu)\le\lim\limits_{m\to\infty}\frac{\mu[0,d\sum\limits_{j=0}^m N^j]}{(d\sum\limits_{j=0}^{m}N^j)^s}=\lim\limits_{m\to\infty}\frac{2^{m+1}}{(\frac{N^{m+1}-1}{N-1}d)^{s}}=(\frac{N-1}{d})^{s}.
\end{eqnarray}
It follows from (\ref{e-3-13}) and (\ref{e-3-15}) that $\mathcal{E}_s^+(\mu)=(\frac{N-1}{d})^{s}$. Thus, $\mathcal{H}^s(K)=(\frac{N-1}{d})^{-s}$ by Theorem \ref{t-3-3}.
\end{proof}

Using the result of Proposition \ref{ex-3-1}, it is easy to compute the Hausdorff measure of the middle Cantor set $K$ satisfying $3K=K\bigcup (K+2)$ associated with its Hausdorff dimension $s=\log_3^2$. That is,
$$\mathcal{H}^s(K)=(\frac{3-1}{2})^{-\log_3^2}=1.$$

\end{document}